\documentclass[11pt]{amsart}
\usepackage{amssymb, latexsym}
\theoremstyle{plain}
\newtheorem{theorem}{Theorem}

\newtheorem {lemma}{Lemma}
\newtheorem{proposition}{Proposition}

\newtheorem*{2'}{Theorem 2'}
\newtheorem*{3'}{Theorem 3'}

\theoremstyle{remark}

\newtheorem*{Remark 1}{Remark 1}
\newtheorem*{Remark 2}{Remark 2}
\newtheorem*{Remark 3}{Remark 3}
\newtheorem*{Remark 4}{Remark 4}

\numberwithin{equation}{section}

\begin{document}

\title [The speed of a general  random walk reinforced by its  history]
{The speed of a general random walk reinforced by its recent history}

\author{Ross G. Pinsky}


\address{Department of Mathematics\\
Technion---Israel Institute of Technology\\
Haifa, 32000\\ Israel}
\email{ pinsky@math.technion.ac.il}

\urladdr{http://www.math.technion.ac.il/~pinsky/}

\date{}

\begin{abstract}

We consider a class of random walks whose increment distributions depend on the average value of the process over its most recent $N$ steps.
We investigate the speed of the process, and in particular, the limiting speed as the ``history window'' $N\to\infty$.

\end{abstract}

\maketitle
\section{Introduction and Statement of Results}\label{intro}
Over the past couple of decades, many papers have been devoted to the study
of edge or vertex reinforced random walks and excited  (also known as ``cookie'') random walks on $\mathbb{Z}$.
These processes have a simple underlying transition mechanism---such as simple symmetric random walk---but this
mechanism is ``reinforced'' or ``excited''  depending on the location of the random walk and its complete  history at that location.
For survey papers which include many references, see \cite{P} and \cite{KZ}.

In  this paper, we consider
random walks on $\mathbb{R}$ with a simpler and very natural
mechanism for reinforcement; namely,  the reinforcement is catalyzed  by the behavior of the
random walk path over a bounded interval of its  history, irrespective of its present location. In fact, we will define
two versions of such a process.
To define these processes, let $N,l\in\mathbb{N}$ with $l\le N$, let $\{P^{(\text{inc})}_i\}_{i=0}^l$ be probability measures on $\mathbb{R}$ with
 finite expectations $\mu_i=\int_{-\infty}^\infty xP^{(\text{inc})}_i(dx)$, and
 let $\{r_i\}_{i=1}^l$ be a sequence. We make the following assumption.
\medskip

\noindent \bf Assumption A.\rm\
The  sequence $\{\mu_i\}_{i=0}^l$ of the expectations corresponding to the measures
$\{P^{(\text{inc})}_i\}_{i=0}^l$
 is strictly increasing, and the sequence $\{r_i\}_{i=1}^l$  satisfies
$$
\begin{aligned}
&r_i<\mu_i<r_{i+1},\ i=1,\cdots, l-1;\\
&\mu_0<r_1\ \ \text{and}\
  r_l<\mu_l.
\end{aligned}
$$

In our notation
for the processes, we suppress the dependence on all the above parameters with the exception of $N$.
One version, the \it instantaneous\rm\ version,  will be denoted by $\{X^{N;I}_n\}_{n=0}^\infty$, while the other
version, the \it delayed\rm\ version, will be denoted by $\{X^{N;D}_n\}_{n=0}^\infty$.
Most of this paper will concern the delayed version, but we will define
the instantaneous version first, because this will make it easier to describe the delayed version.
For convenience, define $r_0=-\infty$ and $r_{l+1}=+\infty$.

The instantaneous version $\{X^{N;I}_n\}_{n=0}^\infty$ is defined as follows. Let $X^{N;I}_0=0$ and let
$\{X^{N;I}_n\}_{n=1}^N$ be distributed like a random walk with increment distribution $P^{(\text{inc})}_{i_0}$, for
some $i_0$. The continuation of the process is defined inductively as follows.
Let $n\ge N+1$ and let  $i$ be such that the process used the distribution $P^{(\text{inc})}_i$ at time $n-1$.
The process looks back at its most recent $N$ steps. If the average value,  $\frac{X^{N;I}_n-X^{N;I}_{n-N}}N$, of those
steps fell in the range $[r_i,r_{i+1})$, then at time $n$ the process jumps with increment distribution
$P^{(\text{inc})}_i$.
However, if     the average value of those steps was strictly less than $r_i$,
then at time $n$ the process jumps with increment distribution
$P^{(\text{inc})}_{i-1}$, while if
 the average value of those steps was larger or equal to $r_{i+1}$
 then
then at time $n$ the process jumps with increment distribution
$P^{(\text{inc})}_{i+1}$.

The delayed version $\{X^{N;D}_n\}_{n=0}^\infty$ is defined similarly, the only difference being that this process  is required to
use any particular jump distribution at least $N$ consecutive times, thereby insuring that
the reinforcement that causes the process to switch from one increment distribution, say $i$, to another increment
distribution is due to the behavior of the process while in the $i$ regime.
Thus, $\{X^{N;D}_n\}_{n=0}^N$ is defined
identically to $\{X^{N;I}_n\}_{n=0}^N$, and for each time $n\ge N+1$, if the jump distribution
used  at time $n-1$ was not used at time $n-N$, then the jump distribution used at time $n-1$  is automatically used again
at time $n$, while otherwise the jump distribution at time $n$ is determined
as it was for the instantaneous version.

We call each version of the process a random walk reinforced by its recent history.
Both versions are  natural models for the fortunes
of various economic commodities, such as stocks, or for the popularity of various social trends,
 which respond positively to recent success and negatively to recent failure.

We call $N$ the \it history window\rm\ and $\{r_i\}_{i=1}^l$ the \it threshold \rm\ levels.
In Assumption B below, we specify a simple condition  to ensure that the processes will almost surely jump an infinite number of times according to  each
of the $l+1$ increment distributions.


In this paper, we  investigate the speeds of these processes.
For the delayed version, it's rather  easy to show that the speed exists almost surely   and is almost surely constant.
\begin{proposition}\label{speed}
Let Assumptions A and Assumption B (given below) hold.
Then the speed
$$
s^D(N,r_1,\cdots, r_l):=\lim_{n\to\infty}\frac{X^{N;D}_n}n
$$
 exists almost surely and is almost surely constant.
\end{proposition}
The proof of  the proposition is embedded in the proof of the main result, Theorem \ref{DD1}, and is noted where it occurs.
The main result concerns  the limiting speed of the delayed version as the history window $N\to\infty$.
Here is the condition we impose  to ensure that the processes will almost surely jump an infinite number of times according to  each
of the $l+1$ increment distributions.

\noindent \bf Assumption B.\rm\
$$
\begin{aligned}
&P_i^{(\text{inc})}\big((-\infty,r_i)\big)>0\ \text{and}\ P_i^{(\text{inc})}\big([r_{i+1},\infty)\big)>0,\ \text{for}\ i=1,\cdots ,l-1;\\
&P_0^{(\text{inc})}([r_1,\infty))>0,\ \ P_l^{(\text{inc})}((-\infty,r_l))>0.
\end{aligned}
$$
(Assumption B is a bit stronger than  necessary to ensure that the process will almost surely jump an infinite number of times according to  each
of the $l+1$ increment distributions, but we use it so as to  simplify the exposition.)

A key  technical  tool that  will be used is  Cram\'er's  large deviations theorem for the empirical mean of an IID sequence.
In order to have this at our disposal, we
need to make a  two-sided exponent moment assumption on the increment distributions $\{P_i^{(\text{inc})}\})_{i=0}^l$.
Let
$$
M_{{P^{(\text{inc})}_i}}(t)=\int_{-\infty}^\infty e^{tx}P_i^{(\text{inc})}(dx)
$$
denote the moment generating function of the distribution $P_i^{(\text{inc})}$.
\medskip

\noindent \bf Assumption C.\rm\
 There exists a  $t_0>0$ such that  $M_{{P^{(\text{inc})}_i}}(\pm t_0)<\infty$, for  $i=0,1,\cdots, l$.
\medskip

Let $I_i(r)$ denote the Legendre-Fenchel transformation for the distribution
$P^{(\text{inc})}_i$, defined by
\begin{equation}\label{Fenchel}
I_i(r)=\sup_{\lambda\in\mathbb{R}}\big(\lambda r-\log M_{{P^{(\text{inc})}_i}}(\lambda)\big),
\ r\in \mathbb{R}.
\end{equation}
We recall several facts about  $I_i$ that we will need and that hold under Assumption C \cite{DZ}.
\begin{equation}\label{fact1}
\begin{aligned}
 &I_i(r)<\infty \ \text{if and only if either}\ r\le \mu_i\ \text{and}\ P_i^{(\text{inc})}(-\infty,r])>0,\ \text{or}\\
 &     r> \mu_i\ \text{and}\ P_i^{(\text{inc})}([r,\infty))>0.
\end{aligned}
\end{equation}
Let $x_i^+=\sup\{x\in\mathbb{R}: I_i(x)<\infty\}$ and
 $x_i^-=\inf\{x\in\mathbb{R}: I_i(x)<\infty\}$.
Then
\begin{equation}\label{fact12}
\begin{aligned}
&I_i(\mu_i)=0;\\
&I_i:[\mu_i,x_i^+)\to[0,\infty) \ \text{is continuous and strictly increasing};\\
&I_i:(x_i^-,\mu_i]\to[0,\infty) \ \text{is continuous and strictly decreasing}.
\end{aligned}
\end{equation}
And we recall an elementary  large deviations result, a version of Cram\'er's theorem, that holds under Assumption C \cite{DZ}:
 if $S^{(i)}_n$ is the sum of $n$ IID random variables
distributed as $P^{(\text{inc})}_i$, and  $P^{(\text{inc})}_i$   satisfies Assumption C,   then
\begin{equation}\label{ld}
\begin{aligned}
&\lim_{n\to\infty}\frac1n\log P(\frac{S^{(i)}_n}n\ge r)=\lim_{n\to\infty}\frac1n\log P(\frac{S^{(i)}_n}n> r)=
-I_i(r),\ \mu_i\le r<x^+_i;\\
&\lim_{n\to\infty}\frac1n\log P(\frac{S^{(i)}_n}n\le r)=\lim_{n\to\infty}\frac1n\log P(\frac{S^{(i)}_n}n< r)=-I_i(r), \ x^-_i<r\le\mu_i.
\end{aligned}
\end{equation}

\medskip

We can now state the main result.
\begin{theorem}\label{DD1}
Let Assumptions A,B, and C hold.
Define
$$
\begin{aligned}
&\Lambda_0=I_0(r_1);\\
&\Lambda_i=I_i(r_{i+1})+\sum_{k=1}^i\big(I_k(r_k)-I_k(r_{k+1})\big),\ 1\le i\le l-1;\\
&\lambda_l=I_l(r_l)+\sum_{k=1}^{l-1}\big(I_k(r_k)-I_k(r_{k+1})\big).
\end{aligned}
$$
If $\max_{0\le i\le l}\Lambda_i$ occurs uniquely at $i=i_0$, then the speed $s^{D}(N,r_1^{(N)},\cdots, r_l^{(N)})$ of
   the delayed process
$\{X^{N;D}_n\}_{n=0}^\infty$ satisfies
$$
\lim_{N\to\infty}s^{D}(N,r_1^{(N)},\cdots, r_l^{(N)})=\mu_{i_0}.
$$
\end{theorem}

\bf \noindent Example.\rm\
The Legendre-Fenchel transformation of the Gaussian distribution $N(\mu,\sigma^2)$ is given
by $I(r)=\frac{(r-\mu)^2}{2\sigma^2}$.
Let $P^{(\text{inc)}}_i\sim N(\mu_i,\sigma_i^2)$. If
$$
\arg\max_{i\in\{0,\cdots, l\}}\Big[\frac{(r_{i+1}-\mu_i)^2}{\sigma_i^2}+\sum_{k=1}^i  \frac{(\mu_k-r_k)^2-(r_{k+1}-\mu_k)^2}{\sigma_k^2} \Big]
$$
occurs uniquely at $i_0$, then the limiting speed for the one-step delayed version is $\mu_{i_0}$.

\medskip

In the instantaneous version, the passage from one regime, say $i$, to a neighboring regime, say $i+1$, will frequently
be accompanied by a number of short time oscillations between the two regimes before the process securely ensconces itself
in the new regime $i+1$. Because of technical difficulties related to these oscillations, we can only
prove a theorem for the limiting speed of the instantaneous version in the case $l=1$.
\begin{theorem}\label{instant}
Let $l=1$ and let Assumptions A,B, and C hold.
The speed  of   the  instantaneous process $\{X^{N;I}_n\}_{n=0}^\infty$ almost surely satisfies
\begin{equation}\label{instantspeed}
\lim_{N\to\infty}\limsup_{n\to\infty}\frac{X^{N;I}_n}n=
\lim_{N\to\infty}\liminf_{n\to\infty}\frac{X^{N;I}_n}n=
\begin{cases} \mu_0,\ \text{if}\ I_0(r_1)>I_1(r_1);\\
\mu_1,\ \text{if}\ I_1(r_1)>I_0(r_1).\end{cases}
\end{equation}
\end{theorem}

In the instantaneous version, define the $N$-dimensional differences  process $\{Z_n^{N;I}\}_{n=0}^\infty$ by
$$
Z^{N;I}_n=(X^{N;I}_{n+1}-X^{N;I}_n, X^{N;I}_{n+2}-X^{N;I}_{n+1},\cdots, X^{N;I}_{n+N}-X^{N;I}_{n+N-1}).
$$
It is easy to see that this is a Markov process.
In \cite{Pin} we studied the speed of the instantaneous version $\{X^{N;I}_n\}_{n=0}^\infty$ under the assumption that the increment distributions
$\{P_i^{(\text{inc})}\}_{i=0}^l$ are all Bernoulli distributions on $\{-1,1\}$; $P^{(\text{inc})}_i\sim\text{Ber}(p_i)$,
so $\mu_i=2p_i-1$. Thus,
those processes lived on $\mathbb{Z}$ and  made only nearest-neighbor jumps.
In that version, we were able to calculate explicitly the invariant measure $\pi^N$ (defined on $\{-1,1\}^N$) of the differences process $\{Z^{N;I}_n\}_{n=0}^\infty$, and this allowed us to  obtain an explicit formula
for the speed $s^I(N,r_1,\cdots, r_l)$. What made the explicit calculation of the invariant distribution possible was the fact that
$\pi^N$ turned out to be constant on the level sets $\{z\in\{-1,1\}^N:\sum_{i=1}^Nz_i=M\}$, for any $M$.
Even in the case that the increment distributions $\{P_i^{(\text{inc})}\}_{i=0}^l$ are all supported on a fixed set
of size three, the explicit calculation of the invariant measure $\pi^N$ of the differences process does not seem possible in general.
Exploiting this explicit formula for the speed $s^I(N,r_1^{(N)},\cdots, r_l^{(N)})$ in the case of Bernoulli increment distributions,
in \cite{Pin} we proved the equivalent of  Theorem \ref{DD1} for the instantaneous version.
The expressions $\{\Lambda\}_{i=0}^l$ in the case of these Bernoulli distributions appear there
in explicit form, but their connection to the Legendre-Fenchel transformation is not mentioned.
The delicate borderline cases, when $\max_{0\le i\le l}\Delta_i$ does not occur uniquely
were also resolved, in each case of which  the limiting speed was  a certain linear combination of the speeds $\{\mu_i\}_{i=0}^l$.
In this paper, we work on exponential scale, via \eqref{ld}, so we cannot handle the borderline cases.

We now turn to the organization of the rest of the paper.
 Theorem \ref{DD1} is proved very quickly  in section \ref{DD1proof}, but this is only after
  a number of  technical propositions are proved in the rather long section \ref{seriesprop}.
 As already noted, the proof of Proposition \ref{speed}  is embedded in the proof of Theorem \ref{DD1}.
 The proof of Theorem \ref{instant} is given in section \ref{instantproof}.

 Here is a rough outline of the idea of the proof of Theorem \ref{DD1}.
Let $\{Y^{N;D}_m\}_{m=0}^\infty$ denote the Markov chain (more specifically, birth and death chain) on $\{0,\cdots, l\}$
that follows the changes of the increment distribution utilized by
the delayed version $\{X_n^{N;D}\}_{n=0}^\infty$ of the  random walk reinforced by its recent history.
Thus, $Y^{N;D}_0=i_0$, since the process $\{X_n^{N;D}\}_{n=0}^\infty$ starts  out using the increment distribution
$P_{i_0}^{(\text{inc})}$. If the first time the process $\{X_n^{N;D}\}_{n=0}^\infty$ changes its increment distribution,
it switches from distribution $P_{i_0}^{(\text{inc})}$ to distribution
$P_j^{(\text{inc})}$ ($j=i_0+1$ or $j=i_0-1$), then $Y^{N;D}_1=j$.
In general, $Y^{N;D}_m=k$, if after switching increment distribution $m$ times, the process
$\{X_n^{N;D}\}_{n=0}^\infty$ is using the increment distribution $P_k^{(\text{inc})}$.
(This Markov chain is introduced after the proof of  Proposition \ref{?} in section \ref{seriesprop}.)
Propositions \ref{key} and \ref{keyprop1}, the first two propositions of section \ref{seriesprop},
 are the key technical results that are used to prove Proposition \ref{transest},
 which   gives tight exponential estimates as $N\to\infty$ on the transition probabilities
of the Markov chain  $\{Y^{N;D}_m\}_{m=0}^\infty$.
Since $\{Y^{N;D}_m\}_{m=0}^\infty$ is a birth and death chain,
its invariant distribution can be written down explicitly in terms of its transition probabilities; thus we obtain
tight exponential estimates on the behavior of  this invariant measure as $N\to\infty$.
Propositions \ref{exptau} and \ref{MG}  calculate respectively  the exponential order as $N\to\infty$ of the expected number of steps made by and   the expected
 distance travelled
by the  delayed version of the random walk reinforced by its recent history
between the time it enters a particular increment distribution regime until it switches to a different increment distribution regime.
The proof of Theorem \ref{DD1} in section \ref{DD1proof} follows easily from Propositions \ref{exptau} and \ref{MG} along with the asymptotic behavior of the invariant measure
for the Markov chain  $\{Y^{N;D}_m\}_{m=0}^\infty$.

\section{A series of propositions}\label{seriesprop}
We will use the following  notation throughout the paper.
$$
\begin{aligned}
&a_N\approx b_N\ \text{means} \
\lim_{N\to\infty}\frac1N\log a_N=\lim_{N\to\infty}\frac1N\log b_N;\\
&a_N\lessapprox b_N \ \text{means}\
\limsup_{N\to\infty}\frac1N\big(\log a_N-\log b_N\big)\le 0;\\
&a_N\lessapprox\lessapprox b_N \ \text{means}\
\limsup_{N\to\infty}\frac1N\big(\log a_N-\log b_N\big)< 0.
\end{aligned}
$$
The random walk with increment distribution $P^{(\text{inc})}_i$ will be denoted by $\{S_n^{(i)}\}_{n=0}^\infty$.
Also, we will use the notation
$$
S^{(i)}_{j,k}=S_k^{(i)}-S^{(i)}_j,\ \text{for}\ 0\le j<k.
$$

In order to reduce the cumbersome notation, we define as follows $Z^{N,i}_n$, for $n\ge1$ and  $1\le i\le l-1$,
with  $1\le j\le i$ and  $i+1\le k\le l$:
\begin{equation}\label{Z}
\begin{aligned}
&Z^{N,i}_n=1,\ \text{if}\\
&\max\big(\frac{S^{(i)}_{(n-1)N,nN}}N,\frac{S^{(i)}_{(n-1)N+1,nN+1}}N,\cdots,\frac{S^{(i)}_{(n-1)N+N-1,nN+N-1}}N\big)\ge r_{i+1}\ \text{and}\\
&\min\big(\frac{S^{(i)}_{(n-1)N,nN}}N,\frac{S^{(i)}_{(n-1)N+1,nN+1}}N,\cdots,\frac{S^{(i)}_{(n-1)N+N-1,nN+N-1}}N\big)\ge r_i;\\
&Z^{N,i}_n=-1,\ \text{if}\\
&\max\big(\frac{S^{(i)}_{(n-1)N,nN}}N,\frac{S^{(i)}_{(n-1)N+1,nN+1}}N,\cdots,\frac{S^{(i)}_{(n-1)N+N-1,nN+N-1}}N\big)< r_{i+1}\ \text{and}\\
&\min\big(\frac{S^{(i)}_{(n-1)N,nN}}N,\frac{S^{(i)}_{(n-1)N+1,nN+1}}N,\cdots,\frac{S^{(i)}_{(n-1)N+N-1,nN+N-1}}N\big)
< r_i;\\
&Z^{N,i}_n=-11,\ \text{if}\\
&\max\big(\frac{S^{(i)}_{(n-1)N,nN}}N,\frac{S^{(i)}_{(n-1)N+1,nN+1}}N,\cdots,\frac{S^{(i)}_{(n-1)N+N-1,nN+N-1}}N\big)
\ge r_{i+1}\ \text{and}\\
&\min\big(\frac{S^{(i)}_{(n-1)N,nN}}N,\frac{S^{(i)}_{(n-1)N+1,nN+1}}N,\cdots,\frac{S^{(i)}_{(n-1)N+N-1,nN+N-1}}N\big)< r_i;\\
&Z^{N,i}_n=0,\ \text{otherwise}.
\end{aligned}
\end{equation}
Note that $\{Z^{N,i}_n\}_{n=1}^\infty$ are identically distributed, and that
each of $\{Z^{N,i}_{2n}\}_{n=1}^\infty$   and $\{Z^{N,i}_{2n-1}\}_{n=1}^\infty$ is an independent sequence.

We begin with two key propositions with rather involved proofs. These propositions  serve as a basis for the rest of the results in this section.
For both  of  them, we will need the FKG correlation inequality in the following form.
Let $W=(W_1,\cdots, W_M)$ be an $\mathbb{R}^M$-valued random variable. Let $f,g:\mathbb{R}^M\to\mathbb{R}$. Then
\begin{equation}\label{FKG}
\begin{aligned}
&Ef(W)g(W)\ge Ef(W)Eg(W), \ \text{if} \ f \ \text{and}\ g\ \text{are either both increasing}\\
& \text{or both decreasing in each of their} \ M\ \text{variables};\\
&Ef(W)h(W)\le Ef(W)Eh(W),\ \text{if one of } \ f \ \text{and}\ g\ \text{is increasing and the other  }\\
&\text{one is decreasing in each of its} \ M\ \text{variables}.
\end{aligned}
\end{equation}
See, for example, \cite{F}.
\begin{proposition}\label{key}
Let $1\le i\le l-1$. Then
\begin{equation}\label{keyequ}
\begin{aligned}
&P(Z^{N,i}_1=1)\approx e^{-NI_i(r_{i+1})};\\
&P(Z^{N,i}_1=-1)\approx e^{-NI_i(r_i)};\\
&P(Z^{N,i}_1=-11)\lessapprox e^{-N\big(I_i(r_i)+I_i(r_{i+1})\big)}.
\end{aligned}
\end{equation}
\end{proposition}
\begin{proof}
We will prove the first and third formulas in  \eqref{keyequ}; the second one is proved analogous to the first.
For the first formula,
 we may assume that
$I_i(r_{i+1})<\infty$, since otherwise,  by \eqref{fact1}, the formula clearly holds.
By \eqref{ld}, we have
\begin{equation}\label{easyupper}
\begin{aligned}
&P(Z^{N,i}_1=1)\le P\big(\max\big(\frac{S^{(i)}_{0,N}}N,\frac{S^{(i)}_{1,N+1}}N,\cdots,\frac{S^{(i)}_{N-1,2N-1}}N\big)
\ge r_{i+1}\big)\le\\
&\sum_{j=0}^{N-1}P(\frac{S^{(i)}_{j,N+j}}N\ge r_{i+1})\approx N e^{-NI_i(r_{i+1})}\approx e^{-NI_i(r_{i+1})}.
\end{aligned}
\end{equation}
Also
\begin{equation}\label{lower}
\begin{aligned}
&P(Z^{N,i}_1=1)= \\
&P\big(\max\big(\frac{S^{(i)}_{0,N}}N,\frac{S^{(i)}_{1,N+1}}N,\cdots,\frac{S^{(i)}_{N-1,2N-1}}N\big)\ge r_{i+1}\big)\times\\
&P\big(\min\big(\frac{S^{(i)}_{0,N}}N,\frac{S^{(i)}_{1,N+1}}N,\cdots,\frac{S^{(i)}_{N-1,2N-1}}N\big)\ge r_i|
\max\big(\frac{S^{(i)}_{0,N}}N,\frac{S^{(i)}_{1,N+1}}N,\cdots,\frac{S^{(i)}_{N-1,2N-1}}N\big)\ge r_{i+1}\big).
\end{aligned}
\end{equation}
By \eqref{ld},
\begin{equation}\label{trivlower}
P\big(\max\big(\frac{S^{(i)}_{0,N}}N,\frac{S^{(i)}_{1,N+1}}N,\cdots,\frac{S^{(i)}_{N-1,2N-1}}N\big)\ge r_{i+1}\big)\ge
P(\frac{S^{(i)}_{0,N}}N\ge r_{i+1})\approx e^{-NI_i(r_{i+1})}.
\end{equation}
The following inequality follows from the FKG correlation inequality \eqref{FKG}.
\begin{equation}\label{posassoc}
\begin{aligned}
&P\big(\min\big(\frac{S^{(i)}_{0,N}}N,\frac{S^{(i)}_{1,N+1}}N,\cdots,\frac{S^{(i)}_{N-1,2N-1}}N\big)\ge r_i|
\max\big(\frac{S^{(i)}_{0,N}}N,\frac{S^{(i)}_{1,N+1}}N,\cdots,\frac{S^{(i)}_{N-1,2N-1}}N\big)\ge r_{i+1}\big)\ge\\
&P\big(\min\big(\frac{S^{(i)}_{0,N}}N,\frac{S^{(i)}_{1,N+1}}N,\cdots,\frac{S^{(i)}_{N-1,2N-1}}N\big)\ge r_i).\\
\end{aligned}
\end{equation}
To see that \eqref{posassoc} follows from \eqref{FKG}, let
$x=(x_1,\cdots, x_{2N-1})\in\mathbb{R}^{2N-1}$, let
$s_{i,j}=\sum_{k=i+1}^jx_k$, for $0\le i<j\le 2N-1$, and define
$$
\begin{aligned}
f(x)=1_{\min\big(\frac{s_{0,N}}N,\frac{s_{1,N+1}}N,\cdots,\frac{s_{N-1,2N-1}}N\big)\ge r_i};\\
g(x)=1_{\max\big(\frac{s_{0,N}}N,\frac{s_{1,N+1}}N,\cdots,\frac{s_{N-1,2N-1}}N\big)\ge r_{i+1}}.
\end{aligned}
$$
 Denote the increments of the random walk $\{S_n^{(i)}\}_{n=0}^\infty$ by $\{W^{(i)}_n\}_{n=1}^\infty$; that is,
$S_n^{(i)}=\sum_{k=1}^nW^{(i)}_k$.
Let $W^{(i)}=(W^{(i)}_1,\cdots, W^{(i)}_{2N-1})$.
Then \eqref{posassoc} is equivalent to $Ef(W^{(i)})g(W^{(i)})\ge Ef(W^{(i)})Eg(W^{(i)})$, and this latter inequality follows from \eqref{FKG}.

From \eqref{posassoc} and
\eqref{ld}  we have
\begin{equation}\label{condlower}
\begin{aligned}
&P\big(\min\big(\frac{S^{(i)}_{0,N}}N,\frac{S^{(i)}_{1,N+1}}N,\cdots,\frac{S^{(i)}_{N-1,2N-1}}N\big)\ge r_i|
\max\big(\frac{S^{(i)}_{0,N}}N,\frac{S^{(i)}_{1,N+1}}N,\cdots,\frac{S^{(i)}_{N-1,2N-1}}N\big)\ge r_{i+1}\big)\ge\\
&1-P\big(\min\big(\frac{S^{(i)}_{0,N}}N,\frac{S^{(i)}_{1,N+1}}N,\cdots,\frac{S^{(i)}_{N-1,2N-1}}N\big)< r_i)\ge
1-NP(\frac{S^{(i)}_{0,N}}N<  r_i)\approx1,\ \text{as}\ N\to\infty.
\end{aligned}
\end{equation}
The first formula in \eqref{keyequ}   now follows from \eqref{easyupper}-\eqref{condlower}.

We now turn to the third formula in  \eqref{keyequ}. We have
\begin{equation}\label{thirdform1}
\begin{aligned}
&P(Z^{N,i}_1=-11)=P\big(\max\big(\frac{S^{(i)}_{0,N}}N,\cdots,\frac{S^{(i)}_{N-1,2N-1}}N\big)\ge r_{i+1}\big)\times\\
&P\big(\min\big(\frac{S^{(i)}_{0,N}}N,\frac{S^{(i)}_{1,N+1}}N,\cdots,\frac{S^{(i)}_{N-1,2N-1}}N\big)< r_i|
\max\big(\frac{S^{(i)}_{0,N}}N,\frac{S^{(i)}_{1,N+1}}N,\cdots,\frac{S^{(i)}_{N-1,2N-1}}N\big)\ge r_{i+1}\big).
\end{aligned}
\end{equation}
By \eqref{ld},
\begin{equation}\label{thirdform2}
\begin{aligned}
&P\big(\max\big(\frac{S^{(i)}_{0,N}}N,\cdots,\frac{S^{(i)}_{N-1,2N-1}}N\big)\ge r_{i+1}\big)\le NP(\frac{S^{(i)}_{0,N}}N\ge r_{i+1})\approx e^{-NI_i(r_{i+1})}.
\end{aligned}
\end{equation}
The  first inequality below follows from the FKG inequality \eqref{FKG} similarly to the way \eqref{posassoc} followed from \eqref{FKG}.
  Using this  and \eqref{ld}, we have
\begin{equation}\label{thirdform3}
\begin{aligned}
&P\big(\min\big(\frac{S^{(i)}_{0,N}}N,\frac{S^{(i)}_{1,N+1}}N,\cdots,\frac{S^{(i)}_{N-1,2N-1}}N\big)< r_i|
\max\big(\frac{S^{(i)}_{0,N}}N,\frac{S^{(i)}_{1,N+1}}N,\cdots,\frac{S^{(i)}_{N-1,2N-1}}N\big)\ge r_{i+1}\big)\le\\
&P\big(\min\big(\frac{S^{(i)}_{0,N}}N,\frac{S^{(i)}_{1,N+1}}N,\cdots,\frac{S^{(i)}_{N-1,2N-1}}N\big)< r_i)\le
NP(\frac{S^{(i)}_{0,N}}N\le r_i)\approx e^{-NI_i(r_i)}.
\end{aligned}
\end{equation}
The third formula in \eqref{keyequ} follows from \eqref{thirdform1}-\eqref{thirdform3}.
\end{proof}

\begin{proposition}\label{keyprop1}
Let $1\le i\le l-1$.
Define
\begin{equation}\label{tau}
\tau^{N,i}=\inf\big\{n\ge0:\frac{S^{(i)}_{n,N+n}}N\not\in[r_i,r_{i+1}) \}.
\end{equation}
Then
\begin{equation}\label{tauformula}
\begin{aligned}
&P(\frac{S^{(i)}_{\tau^{N,i},N+\tau^{N,i}}}N< r_i)\approx e^{-N\big(I_i(r_i)-I_i(r_{i+1})\big)^+};\\
&P(\frac{S^{(i)}_{\tau^{N,i},N+\tau^{N,i}}}N\ge r_{i+1})\approx e^{-N\big(I_i(r_{i+1})-I_i(r_i)\big)^+}.
\end{aligned}
\end{equation}
\end{proposition}

\begin{proof}
By  Assumption B, it follows that $\tau^{N,i}<\infty$ a.s. Also, by Assumption B and \eqref{fact1}, it follows that
 $I_i(r_k)$ and $I_i(r_j)$ are finite.
Assume without loss of generality that $I_i(r_i)\ge I_i(r_{i+1})$. If $I_i(r_i)> I_i(r_{i+1})$,
then it suffices to prove the first formula in \eqref{tauformula} since the two terms on the left hand side
of \eqref{tauformula} add up to one.
If $I_i(r_j)= I_i(r_k)$, then the proofs of the two formulas in \eqref{tauformula} are almost identical. Thus, in this case too we will prove only the first
formula.
Suppressing the dependence on $N$, let
\begin{equation}\label{stopping}
\sigma_i^{(e)}=\inf\big\{2n\ge 2: Z_{2n}^{N,i}\neq0\}\big\},\ \ \sigma_i^{(o)}=\inf\big\{2n-1\ge 1, Z_{2n-1}^{N,i}\neq0\}\big\}.
\end{equation}
Using Proposition \ref{key} and the fact that
each of $\{Z^{N,i}_{2n}\}_{n=1}^\infty$   and $\{Z^{N,i}_{2n-1}\}_{n=1}^\infty$ is an IID sequence,
it follows that
\begin{equation}\label{likethmform}
\begin{aligned}
&P( Z_{\sigma_{i}^{(*)}}^{N,i}=-1)\approx e^{-N\big(I_i(r_i)-I_i(r_{i+1})\big)},\\
&P( Z_{\sigma_{i}^{(*)}}^{N,i}=-11)\lessapprox e^{-N I_i(r_i)},\\
&\text{both when}\ \sigma_{i}^{(*)}=\sigma_{i}^{(e)}\ \text{and when}\ \sigma_{i}^{(*)}=\sigma_{i}^{(o)}.
\end{aligned}
\end{equation}
Now
$$
\Big\{\frac{S^{(i)}_{\tau^{N,i},N+\tau^{N,i}}}N< r_i\Big\}\subset\big\{Z^{N,i}_{\sigma_{i}^{(e)}}\in\{-1,-11\}\big\}
\cup\big\{Z^{N,i}_{\sigma_{i}^{(0)}}\in\{-1,-11\}\big\};
$$
thus, it follows from \eqref{likethmform} that
\begin{equation}\label{upperbound}
P\big(\frac{S^{(i)}_{\tau^{N,i},N+\tau^{N,i}}}N< r_i\big)\lessapprox e^{-N\big(I_i(r_i)-I_i(r_{i+1})\big)}.
\end{equation}

To prove an inequality in the other direction, let $a_N=P(Z_1^{(N,i)}=-1)$ and
$b_N=P(Z_1^{(N,i)}\in\{1,-11\})$, where we have suppressed the dependence on $i$.
From Proposition \ref{key},
\begin{equation}\label{aNbN}
a_N\approx e^{-NI_i(r_i)},\ \ b_N\approx e^{-NI_i(r_{i+1})}.
\end{equation}
We have for any positive integer $M$,
\begin{equation}
\big\{\frac{S^{(i)}_{\tau^{N,i},N+\tau^{N,i}}}N< r_i\Big\}\supset \big(\cap_{n=1}^{2M}
\big\{Z_n^{N,i}\in\{0,-1\}\big\}\big)\cap\big(\cup_{n=1}^M\{Z_{2n}^{N,i}=-1\}\big).
\end{equation}
Thus,
\begin{equation}\label{estimate}
\begin{aligned}
&P(\frac{S^{(i)}_{\tau^{N,i},N+\tau^{N,i}}}N< r_i)\ge
P(\cup_{n=1}^M\{Z_{2n}^{N,i}=-1\})\times\\
&P\big(\cap_{n=1}^{2M}
\big\{Z_n^{N,i}\in\{0,-1\}\big\}\big|\cup_{n=1}^M\{Z_{2n}^{N,i}=-1\}\big).
\end{aligned}
\end{equation}
Since $\{Z_{2n}^{N,i}\}_{n=1}^M$
are IID, it follows that
\begin{equation}\label{unconditioned}
P(\cup_{n=1}^M\{Z_{2n}^{N,i}=-1\})=1-(1-a_N)^M.
\end{equation}
From the definitions, it follows that
\begin{equation}\label{forfirstevent}
\{Z^{N,i}_n\in\{0,-1\}\}=\cap_{m=(n-1)N}^{nN-1}\{\frac{S^{(i)}_{m,N+m}}N<r_{i+1}\}
\end{equation}
and
\begin{equation}\label{alsofirstevent}
\{Z_{2n}^{N,i}=-1\}=\big(\cap_{m=(2n-1)N}^{2nN-1}\{\frac{S_{m,N+m}^{(i)}}N<r_{i+1}\}\big)
\cap\big(\cup_{m=(2n-1)N}^{2nN-1}\{\frac{S_{m,N+m}^{(i)}}N< r_i\}\big).
\end{equation}
From \eqref{forfirstevent} and \eqref{alsofirstevent}, along with the FKG inequality \eqref{FKG}, we have
\begin{equation}\label{poscorragain}
P\big(\cap_{n=1}^{2M}
\big\{Z_n^{N,i}\in\{0,-1\}\big\}\big|\cup_{n=1}^M\{Z_{2n}^{N,i}=-1\}\big)\ge P\big(\cap_{n=1}^{2M}
\big\{Z_n^{N,i}\in\{0,-1\}\big\}\big).
\end{equation}
To see this, let
$x=(x_1,\cdots, x_{(2M+1)N-1})\in\mathbb{R}^{(2M+1)N-1}$, let
$s_{i,j}=\sum_{k=i+1}^jx_k$, for $0\le i<j\le (2M+1)N-1$, and define
$$
\begin{aligned}
&f(x)=1_{\max\big(\frac{s_{m,N+m}}N:\ m\in\cup_{n=1}^M\{(2n-1)N,\cdots, 2nN-1\}\big)<r_{i+1}}\times\\
&\prod_{n=1}^M1_{\min\big(\frac{s_{m,N+m}}N:\ k\in\{(2n-1)N,\cdots, 2nN-1\}\big)<r_i};\\
&g(x)=1_{\max\big(\frac{s_{m,N+m}}N:\ m\in\cup_{n=1}^{2M}\{(n-1)N,\cdots, nN-1\}\big)<r_{i+1}}.
\end{aligned}
$$
 Denote the increments of the random walk $\{S_n^{(i)}\}_{n=0}^\infty$ by $\{W^{(i)}_n\}_{n=1}^\infty$,  and let
 $W^{(i)}=(W^{(i)}_1,\cdots, W^{(i)}_{(2M+1)N-1})$.
Then \eqref{poscorragain} is equivalent to $Ef(W^{(i)})g(W^{(i)})\ge Ef(W^{(i)})Eg(W^{(i)})$, and this latter inequality follows from \eqref{FKG}.

Similarly, the FKG inequality \eqref{FKG} gives
$$
\begin{aligned}
P\big(\cap_{n=1}^{2M}
\big\{Z_n^{N,i}\in\{0,-1\}\big\}\big)\ge \big(P(
Z_n^{N,i}\in\{0,-1\})\big)^{2M}=
(1-b_N)^{2M}.
\end{aligned}
$$
Thus,
\begin{equation}\label{conditioned}
\begin{aligned}
&P\big(\cap_{n=1}^{2M}
\big\{Z_n^{N,i}\in\{0,-1\}\big\}\big|\cup_{n=1}^M\{Z_{2n}^{N,i}=-1\}\big)\ge(1-b_N)^{2M}.
\end{aligned}
\end{equation}

Now choose $M=[\frac1{b_N}]$. We consider the two cases
 $I_i(r_i)>I_i(r_{i+1})$ and $I_i(r_i)=I_i(r_{i+1})$ separately. We first consider the former case.
Note that $\lim_{N\to\infty}\frac{a_N}{b_N}=0$. Since $1-a_N\le e^{-a_N}$,   from \eqref{unconditioned},
\begin{equation}\label{unconditionedagain}
P(\cup_{n=1}^{[\frac1{b_N}]}\{Z_{2n}^{N,i}=-1\})\ge1-e^{-a_N[\frac1{b_N}]}\ge\frac{a_N}{2b_N},\ \text{for large}\ N.
\end{equation}
From \eqref{conditioned},
\begin{equation}\label{conditionedagain}
\liminf_{N\to\infty}P\big(\cap_{n=1}^{2[\frac1{b_N}]}
\big\{Z_n^{N,i}\in\{0,-1\}\big\}\big|\cup_{n=1}^{[\frac1{b_N}]}\{Z_{2n}^{N,i}=-1\}\big)\ge e^{-2}.
\end{equation}
From \eqref{aNbN},
\eqref{estimate}, \eqref{unconditionedagain} and \eqref{conditionedagain}, we conclude that
\begin{equation}\label{lowerbound}
P\big(\frac{S^{(i)}_{\tau^{N,i},N+\tau^{N,i}}}N<r_j\big)\gtrapprox e^{-N\big(I_i(r_i)-I_i(r_{i+1})\big)}.
\end{equation}

Now consider the case  $I_i(r_i)=I_i(r_{i+1})$.
Then similar to \eqref{unconditionedagain}, we have
\begin{equation}\label{unconditionedagain2}
P(\cup_{n=1}^{[\frac1{b_N}]}\{Z_{2n}^{N,i}=-1\})\ge1-e^{-a_N[\frac1{b_N}]}\ge\min(c, \frac{a_N}{2b_N}),\ \text{for some}\ c>0.
\end{equation}
Then from  \eqref{aNbN},
\eqref{estimate}, \eqref{unconditionedagain2}  \eqref{conditionedagain} and the fact that $a_N\approx b_N$,  we obtain \eqref{lowerbound}.
The first formula in \eqref{tauformula} follows from \eqref{upperbound} and \eqref{lowerbound}.
\end{proof}

Recall the process $\{Y^{N;D}_m\}_{m=0}^\infty$   mentioned at the end of section  \ref{intro}; it
denotes the Markov processes that follows the changes of the increment distribution utilized by
the delayed version $\{X_n^{N;D}\}_{n=0}^\infty$    of the  random walk reinforced by its recent history.
We denote the transitions for $\{Y^{N;D}_m\}_{m=0}^\infty$ by
$$
p^{N;D}_{i,j}=P(Y^{N;D}_{m+1}=j|Y^{N;D}_m=i),\ i,j\in\{0,\cdots, l\}, \ j=i\pm1.
$$

Using Proposition \ref{keyprop1}, the following estimates on these transition probabilities are almost immediate.
\begin{proposition}\label{transest}
\begin{equation}\label{transitions}
\begin{aligned}
&p^{N;D}_{i,i+1}\approx e^{-N\big(I_i(r_{i+1})-I_i(r_i)\big)^+};\ i\in\{1,\cdots, l-1\};\\
&p^{N;D}_{i,i-1}\approx e^{-N\big(I_i(r_i)-I_i(r_{i+1})\big)^+};\ i\in\{1,\cdots, l-1\};\\
&p^{N;D}_{0,1}=p^{N;D}_{l,l-1}=1.\\
\end{aligned}
\end{equation}
\end{proposition}
\begin{proof}
The third line in \eqref{transitions} follows by definition.
Noting that
$$
p^{N;D}_{i,i+1}=P(\frac{S^{(i)}_{\tau^{N,i},N+\tau^{N,i}}}N\ge r_{i+1}),\ \
p^{N;D}_{i,i-1}=P(\frac{S^{(i)}_{\tau^{N,i},N+\tau^{N,i}}}N< r_i),\ \
$$
the first two lines of \eqref{transitions}
 follow  from Proposition \ref{keyprop1}.
\end{proof}

Denote the invariant distribution of the Markov chain $\{Y^{N;D}_m\}_{m=0}^\infty$
on $\{0,\cdots, l\}$
by $\nu^{N;D}$.
The Markov chain  $\{Y^{N;D}_m\}_{m=0}^\infty$ is a birth and death process, thus reversible, so its invariant distribution
can be calculated explicitly, via
 the detailed balance equations: $\nu^{N;D}(i)p^{N;D}_{i,i+1}=\nu^{N;D}(i+1)p^{N;D}_{i+1,i}$, \ $i=0,\cdots, l-1$.
As is well-known, one has
\begin{equation}\label{muhatdist}
\begin{aligned}
&\Pi_N\nu^{N,D}(0)=1;\\
&\Pi_N\nu^{N,D}(k)=\prod_{i=1}^k\frac{p^{N;D}_{i-1,i}}{p^{N;D}_{i,i-1}},\ k=1,\cdots, l,\\
&\text{where}\ \Pi_N=1+\sum_{k=1}^l\prod_{i=1}^k\frac{p^{N;D}_{i-1,i}}{p^{N;D}_{i,i-1}}.
\end{aligned}
\end{equation}

Recall the definition of  $\tau^{N,i},\ 1\le i\le l-1$,
from  \eqref{tau}.
Define
$$
\tau^{N,0}=\inf\big\{n\ge0:\frac{S^{(0)}_{n,N+n}}N\ge r_1 \};\ \ \
 \tau^{N,l}=\inf\big\{n\ge0:\frac{S^{(l)}_{n,N+n}}N< r_l \}.
 $$
 Anytime the delayed version of the random walk reinforced by its
recent history  switches to regime $i$, the  number of steps during which it will operate in this regime before moving
to a different regime is distributed as $\tau^{N,i}+N$,
and the distance it travelled  between its entrance into regime $i$ and its exit to another regime
is distributed as $S^{(i)}_{\tau^{N,i}+N}$. The next two propositions calculate the expected values of these two distributions.
\begin{proposition}\label{exptau}
\begin{equation}\label{exptauformula}
\begin{aligned}
&E\tau^{N,i}\approx e^{N\min\big(I_i(r_i),\thinspace I_i(r_{i+1})\big)},\ 1\le i\le l-1;\\
&E\tau^{N,l,l}\approx e^{NI_l(r_l)};\ \  E\tau^{N,0}\approx e^{NI_0(r_1)}.\\
\end{aligned}
\end{equation}
\end{proposition}
\begin{proof}
Let $1\le i\le l-1$. Using the notation from the  proof of Proposition \ref{keyprop1}, for any positive integer $L$, we have
\begin{equation}\label{tausigma}
\begin{aligned}
&\{\tau^{N,i}\ge 2LN\}=\{Z^{N,i}_n=0,\ \text{for all}\ n=1,\cdots, 2L\}=\\
&\{\sigma^{(e)}_{i}>2L,\sigma^{(o)}_{i}>2L-1\}.
\end{aligned}
\end{equation}
Since $\sigma^{(e)}_{i}$ and $\sigma^{(o)}_{i}+1$ have the same distribution, it follows that
\begin{equation}\label{22L}
 P(\tau^{N,i}\ge 2LN)\le P(\sigma^{(e)}_{i}>2L),\ L\ge1.
\end{equation}
We have
\begin{equation}\label{larger2LN}
\sum_{L=0}^\infty  P(\tau^{N,i}\ge 2LN)\ge\sum_{m=0}^\infty(\frac m{2N}+1)P(\tau^{N,i}=m)=1+\frac1{2N}E\tau^{N,i}.
\end{equation}
From the definition of $\sigma^{(e)}_{i}$ along with Proposition \ref{key},
$\sigma^{(e)}_{i}$ is distributed according to a geometric
distribution with parameter $p\approx  e^{-N\min\big(I_i(r_i),\thinspace I_i(r_{i+1})\big)}$; thus,
$E\sigma^{(e)}_{i}\approx e^{N\min\big(I_i(r_i),\thinspace I_i(r_{i+1})\big)}$.
Consequently,
\begin{equation}\label{larger2L}
\sum_{L=0}^\infty P(\sigma^{(e)}_{i}>2L)\le \sum_{L=1}^\infty P(\sigma^{(e)}_{i}\ge L)=E\sigma^{(e)}_{i}
\approx e^{N\min\big(I_i(r_i),\thinspace I_i(r_{i+1})\big)}.
\end{equation}
From \eqref{22L}-\eqref{larger2L}, we obtain
\begin{equation}\label{upperbd}
E\tau^{N,i}\lessapprox e^{N\min\big(I_i(r_i),\thinspace I_i(r_{i+1})\big)}.
\end{equation}

From  Proposition \ref{key} and  the definition of
$\sigma^{(e)}_{i}$ and $\sigma^{(o)}_{i}$, we have for any $\epsilon>0$ and sufficiently large $N$,
\begin{equation}\label{sigmasinequ}
\begin{aligned}
&P(\sigma^{(e)}_{i}>2L,\sigma^{(o)}_{i}>2L-1)=P(Z^{N,i}_n=0, \ \text{for}\ n=1,\cdots, 2L)\ge\\
&1-2LP(Z^{N,i}_1\neq0)\ge 1-2Le^{\epsilon N-N\min\big(I_i(r_i),\thinspace I_i(r_{i+1})\big)}.
\end{aligned}
\end{equation}
Letting $L_{N,\epsilon}=[e^{-2\epsilon N+N\min\big(I_i(r_i),\thinspace I_i(r_{i+1})\big)}]$, it follows from
\eqref{tausigma} and \eqref{sigmasinequ} that
$\lim_{N\to\infty}P(\tau^{N,i}\ge 2L_{N,\epsilon}N)=1$. Since $\epsilon>0$ is arbitrary, it follows that
\begin{equation}\label{lowerbd}
E\tau^{N,i}\gtrapprox e^{N\min\big(I_i(r_i),\thinspace I_i(r_{i+1})\big)}.
\end{equation}
The first formula in \eqref{exptauformula} follows from \eqref{upperbd} and \eqref{lowerbd}.

The statements of Proposition \ref{key} and Proposition \ref{keyprop1} involve certain two-sided hitting  times related
to  a random walk with increment
distribution $P_i^{(\text{inc})}$, with $1\le i\le l-1$. Similar one-sided results could have been written down for $i=0$ and
$i=l$.
We refrained from including them in order not to incur the necessity of additional notation and an additional analogous proof.
The second  formula in  \eqref{exptauformula} is proved similarly to the first formula using the corresponding
one-sided hitting times.
\end{proof}
\begin{proposition}\label{MG}
\begin{equation}\label{Stau}
ES^{(i)}_{\tau^{N,i}+N}=\mu_i(E\tau^{N,i}+N),\ 0\le i\le l.
\end{equation}
\end{proposition}
\begin{proof}
Let $\{W_n\}_{n=1}^\infty$ be iid random variables distributed according to $P_i^{(\text{inc})}$ and consider  the filtration
 $\mathcal{F}_n=\sigma\Big(W_1,\cdots, W_n\Big), \ n\ge1$.
 We can write  $S^{(i)}_n=\sum_{j=1}^n W_j$.
Now $M_{n+N}:=S^{(i)}_{n+N}-(n+N)\mu_i$, $n\ge0$, is a martingale with respect to
$\{\mathcal{F}_{n+N}\}_{n=0}^\infty$. Note that $N+\tau^{N,i}$ is a stopping time with respect to $\{\mathcal{F}_{n+N}\}_{n=0}^\infty$.
 So by Doob's optional sampling theorem,
$$
ES_{(\tau^{N,i}+N)\wedge L}-\mu_iE((\tau^{N,i}+N)\wedge L)=0,\ L\ge0.
$$
Letting $L\to\infty$ and using  \eqref{exptauformula}, we obtain  \eqref{Stau}.
\end{proof}

\section{Proof of Theorem \ref{DD1}}\label{DD1proof}
Recall that $\nu^{N,D}$ denotes the invariant distribution  of
the processe $\{Y^{N;D}_m\}_{m=0}^\infty$.
By the ergodic theorem,   as $m\to\infty$ the asymptotic
 proportion of switches of the  process  $\{Y^{N;D}_m\}_{m=0}^\infty$ for the  delayed process
to the regime $i$ is $\nu^{N,D}(i)$.
As noted before Proposition \ref{exptau},
anytime the delayed   version of the random walk reinforced by its recent history  switches to regime $i$, the  number of steps during which it will operate in this regime before moving
to a different regime is distributed as $\tau^{N,i}+N$,
and the distance travelled by the process between its entrance into regime $i$ and its exit to another regime
is distributed as $S^{(i)}_{\tau^{N,i}+N}$.
Also, this random number of steps and this random distance travelled are
independent of the random number of steps the process spent and the random distance it travelled in any regime  in the past before the present entrance into
regime $i$.
From these observations,  it is standard to deduce that the speed $s^D(N,r_1,\cdots, r_l)$ , defined in Proposition \ref{speed},
 exists almost surely and is almost surely given by the constant
 \begin{equation}\label{explicitspeed}
s^D(N,r_1,\cdots, r_l)=\frac{\sum_{i=0}^l\nu^{N,D}(i)ES^{(i)}_{\tau^{N,i}+N}}{\sum_{i=0}^l\nu^{N,D}(i)\big(E\tau^{N,i}+N\big)}.
 \end{equation}
 This proves Proposition \ref{speed}.

By Propositions
\ref{exptau} and \ref{MG},
\begin{equation}\label{keyterm1}
\begin{aligned}
&ES^{(i)}_{\tau^{N,i}+N}\approx\mu_ie^{N\min\big(I_i(r_i),I_i(r_{i+1})\big)},\ \text{for}\ 1\le i\le l-1;\\
&ES^{(0)}_{\tau^{N,0}+N}\approx\mu_0 e^{NI_0(r_1)},\
ES^{(l)}_{\tau^{N,l}+N}\approx\mu_l e^{NI_l(r_l)}.
\end{aligned}
\end{equation}
From \eqref{muhatdist} and Proposition \ref{transest}, we have
\begin{equation}\label{keyterm2}
\begin{aligned}
&\nu^{N;D}(0)\approx\frac1{\Pi_N};\\
&\nu^{N;D}(1)\approx\frac1{\Pi_N}e^{N\big(I_1(r_1)-I_1(r_2)\big)^+};\\
&\nu^{N;D}(i)\approx\frac1{\Pi_N}e^{N\big(I_1(r_1)-I_1(r_2)\big)^+}\prod_{k=2}^ie^{N\Big(\big(I_k(r_k)-I_k(r_{k+1})\big)^+
-(I_{k-1}(r_k)-I_{k-1}(r_{k-1})\big)^+\Big)},\\
& 1\le i\le l-1;\\
&\nu^{N;D}(l)\approx\frac1{\Pi_N}e^{N\big(I_1(r_1)-I_1(r_2)\big)^+}\prod_{k=2}^{l-1}e^{N\Big(\big(I_k(r_k)-I_k(r_{k+1})\big)^+
-(I_{k-1}(r_k)-I_{k-1}(r_{k-1})\big)^+\Big)}\times\\
& e^{-N\big(I_{l-1}(r_l)-I_{l-1}(r_{l-1})\big)^+}.
\end{aligned}
\end{equation}
Noting that
$\big(I_k(r_k)-I_k(r_{k+1})\big)^+-(I_k(r_{k+1})-I_k(r_k)\big)^+=I_k(r_k)-I_k(r_{k+1})$
and recalling the definition of $\{\Lambda_i\}_{i=0}^l$ in the statement of Theorem \ref{DD1},
it follows from \eqref{keyterm1} and \eqref{keyterm2} that
\begin{equation}\label{keyterm12}
\nu^{N,D}(i)ES^{(i)}_{\tau^{N,i}+N}
\approx
\frac1{\Pi_N}\mu_ie^{N\Lambda_i},\ 0\le i\le l.
\end{equation}
Substituting \eqref{keyterm12} into the second equation in \eqref{explicitspeed},
recalling from Propositions \ref{exptau} and \ref{MG} that
$\frac{ES^{(i)}_{\tau^{N,i}+N}}{E\tau^{N,i}+N}\approx\mu_i$,
 and letting $N\to\infty$ proves the theorem.

\hfill $\square$





\section{Proof of Theorem \ref{instant}}\label{instantproof}
For the proof of Theorem \ref{instant}, we need the following lemma.
\begin{lemma}\label{lemma1}
Let $\{Z_n\}_{n=1}^\infty$ be IID random variables satisfying $EZ_1=\mu$ and let $S_n=\sum_{i=1}^nZ_i$.
Then for every $r<\mu$,
\begin{equation}\label{alwaysabove}
P(\frac{S_n}n\ge r,\  n=1,2,\cdots)>0.
\end{equation}
\end{lemma}
\begin{proof}
By the strong law of large numbers, $\lim_{n\to\infty}\frac{S_n}n=\mu$ a.s.
Thus, for every $r<\mu$, there exists an $N_r$ such that
$P(\frac{S_n}n\ge r,\  n>N_r)>0$.
Clearly, $P(\frac{S_n}n\ge r,\  n=1,\cdots, N_r)>0$.
Since the events $\{\frac{S_n}n\ge r\ n=1,\cdots, N_r\}$ and
$\{\frac{S_n}n\ge r,\  n>N_r\}$ are positively correlated, we have
$$
\begin{aligned}
&P(\frac{S_n}n\ge r,\  n=1,2,\cdots)=\\
&P(\frac{S_n}n\ge r,\  n=1,\cdots, N_r)P(\frac{S_n}n\ge r,\ n>N_r|\frac{S_n}n\ge r,\ n=1,\cdots, N_r)>0.
\end{aligned}
$$
\end{proof}
We now turn to the proof of the theorem.

\it\noindent Proof of Theorem \ref{instant}.\rm\ Without loss of generality, assume that $I_0(r_1)<I_1(r_1)$.
Since clearly $\lim_{N\to\infty}\limsup_{n\to\infty}\frac{X^{N;I}_n}n\le\mu_1$, what we need to
 prove is  that
\begin{equation}\label{instantspeedmu1}
\lim_{N\to\infty}\liminf_{n\to\infty}\frac{X^{N;I}_n}n=\mu_1.
\end{equation}
Define
\begin{equation}\label{calwaysabove}
c:=P(\frac{S^{(1)}_n}n\ge r_1,\ n=1,2,\cdots)>0,
\end{equation}
 where the positivity of $c$ follows from Lemma \ref{lemma1}.
Without loss  of generality, we will start  the instantaneous process $\{X^{N;I}_n\}_{n=0}^\infty$  in the $P_0^{(inc)}$-mode.
The process will eventually switch to the $P_1^{(inc)}$-mode, then switch back to the $P_0^{(inc)}$, etc.

Let $T^{N,1}_m, m\ge1,$ denote the number of steps the instantaneous process spends in the $P_1^{(inc)}$-mode during its $m$th session in that mode, and let
 $T_m^{N,0}, m\ge1,$ denote the number of steps the instantaneous
process spends in the $P^{(\text{inc})}_0$-mode during its $m$th session in that mode.

 Clearly $T_m^{N,0}$, for any $m\ge1$, is stochastically dominated by $N+\tau^{N,0}$, where $\tau^{N,0}$ is as in \eqref{tau}. (There is equality for $m=1$.)

The event that for all  $j=1,\cdots, N$, the average value of the first $j$ steps of a
$P_1^{(inc)}$-random walk is    greater or equal to $r_1$ has probability greater than $c$. Thus, with probability greater than
$c$, the instantaneous process will spend at least $N$ steps in the $P_1^{(inc)}$-mode during any  session in that mode.
It follows then that $T^{N,1}_m$, for any $m\ge1$, stochastically dominates
$(N+\tau^{N,1})\text{Ber}(c)$, where $\tau^{N,1}$ is as in \eqref{tau}, $\text{Ber}(c)$ denotes a Bernoulli random variable
with probability $c$ of being equal to 1 and probability $1-c$ of being equal to 0, and $\tau^{N,1}$ and $\text{Ber}(c)$
are independent. We note that there are two reasons that $T^{N,1}_m$ stochastically dominates
$(N+\tau^{N,1})\text{Ber}(c)$. One is that the  probability of the event described above is greater than $c$. The other
is that  $\tau^{N,1}$,  the number of steps the delayed process remains in the $P_1^{(inc)}$-mode
after its first $N$ steps in that mode, is stochastically dominated by the random variable
$T^{N,1}_m-N$ when this latter random variable is conditioned on the  event described above.
The reason for this latter domination is that whereas the first $N$ steps of the delayed process have the
distribution $\{S^{(i)}_j\}_{j=1}^N$, the first $N$ steps of the instantaneous process conditioned on the event described above
has the distribution $\{S^{(i)}_j\}_{j=1}^N$, conditioned on $\{\frac{S^{(1)}_j}j\ge r_1,\ j=1,2,\cdots, N\}$,
and by the
 FKG inequality \eqref{FKG}, the distribution $\{S^{(i)}_j\}_{j=1}^N$, conditioned on $\{\frac{S^{(1)}_j}j\ge r_1,\ j=1,2,\cdots, N\}$ dominates the distribution $\{S^{(i)}_j\}_{j=1}^N$.

The fraction of steps that the instantaneous process spends in the $P_1^{(inc)}$-mode after $m$ sessions in each mode is given by
\begin{equation}\label{asympfrac}
\frac{\sum_{k=1}^m T^{N,1}_k}{\sum_{k=1}^m(T^{N,1}_k+T^{N,0}_k)}.
\end{equation}
By
the above noted stochastic domination,
we can define on one and the same space $\{T^{N,1}_k\}_{k=1}^\infty$ and $\{T^{N,0}_k\}_{k=1}^\infty$
along with $\{\tau^{N,i}_k\}_{k=1}^\infty$, $i=0,1$, and $\{\text{Ber}(c)_k\}_{k=1}^\infty$, where these last three sequences are
mutually independent   IID sequences distributed respectively as  $\tau^{N,i}, i=0,1$, and
$\text{Ber}(c)$, such that
\begin{equation}\label{asympfracest}
\frac{\sum_{k=1}^m T^{N,1}_k}{\sum_{k=1}^m(T^{N,1}_k+T^{N,0}_k)}\ge
\frac{\sum_{k=1}^m (N+\tau^{N,1}_k)\text{Ber}(c)_k}{\sum_{k=1}^m(N+\tau^{N,1}_k)\text{Ber}(c)_k+\tau^{N,0}_k}\ \ \text{a.s.}
\end{equation}
By the strong law of large numbers,
\begin{equation}\label{rhs}
\lim_{m\to\infty}\frac{\sum_{k=1}^m (N+\tau^{N,1}_k)\text{Ber}(c)_k}{\sum_{k=1}^m(N+\tau^{N,1}_k)\text{Ber}(c)_k+\tau^{N,0}_k}=\frac{c(N+E\tau^{N,1})}{c(N+E\tau^{N,1})+E\tau^{N,0}}\ \  \text{a.s.}
\end{equation}
By Proposition \ref{exptau} (with $l=1$) and the assumption that $I_0(r_1)<I_1(r_1)$, it follows that
$\lim_{N\to\infty}\frac{E\tau^{N,1}}{E\tau^{N,0}}=\infty$. Using this with \eqref{asympfracest} and \eqref{rhs}, we conclude that the asymptotic fraction of steps that the instantaneous process spends in the
 $P_1^{(inc)}$-mode  satisfies
$$
\lim_{N\to\infty}\liminf_{m\to\infty}\frac{\sum_{k=1}^m T^{N,1}_k}{\sum_{k=1}^m(T^{N,1}_k+T^{N,0}_k)}=1\ \text{a.s.}
$$
From this we conclude that \eqref{instantspeedmu1} holds.
\hfill$\square$


\begin{thebibliography}{99}
\bibitem{DZ}
 Dembo, A. and  Zeitouni, O., \emph{Large Deviations Techniques and Applications} 2nd edition, Springer-Verlag, New York, (1998).

\bibitem{D} Durrett, R. \emph{Probability: theory and examples}. Third edition, Brooks/Cole (2005).

\bibitem{F} Fortuin, C., Ginibre, J. and Kasteleyn, P. \emph{Correlation inequalities on some partially ordered sets}, Comm. Math. Phys., \textbf{22}, (1971), 89-103.


\bibitem{KZ} Kosygina, E. and   Zerner M. \emph{Excited random walks: results, methods, open problems},
Bull. Inst. Math. Acad. Sin. (N.S.)   \textbf{8} (2013),  105-157.

\bibitem{P}
 Pemantle, R.  \emph{A survey of random processes with reinforcement}, Probab. Surv. \textbf{4} (2007), 1-79.

\bibitem{Pin} Pinsky, R. \emph{The speed of a random walk excited by its recent history}, Adv. in Appl. Probab. 48 (2016),  215-234.
\end{thebibliography}
\end{document}